\documentclass{amsart}
\usepackage{amsfonts}
\usepackage{amssymb}
\usepackage{amsthm}
\usepackage{hyperref}

\newtheorem{theorem}{Theorem}[section]

\newtheorem{proposition}[theorem]{Proposition}
\newtheorem{corollary}[theorem]{Corollary}
\newtheorem{remark}[theorem]{Remark}

\title{Reflection length in non-affine Coxeter groups}
\author{Kamil Duszenko}
\address{Instytut Matematyczny, Uniwersytet Wroc{\l}awski,
plac Grunwaldzki 2/4, 50-384 Wroc{\l}aw, Poland}
\email{Kamil.Duszenko@math.uni.wroc.pl}
\subjclass[2000]{20F55, 20F67}
\thanks{The author was partially supported by a MNiSW grant N N201 541738.}

\begin{document}

\begin{abstract}
The reflection length of an element of a Coxeter group is the minimal number
of conjugates of the standard generators whose product is equal to that element.
In this paper we prove the conjecture of McCammond and Petersen that reflection
length is unbounded in any non-affine Coxeter group. Among the tools used,
the construction of word-hyperbolic quotients of all minimal non-affine
Coxeter groups might be of independent interest.
\end{abstract}

\maketitle

\section{Introduction}

Let $(W,S)$ be a Coxeter group and let $R=\{wsw^{-1}:s\in S,w\in W\}$
be the set of reflections. The reflection length of an element $w\in W$
is defined as
\begin{equation}
\label{refllengthdef}
||w||_R=\min\{n:w=r_1r_2\ldots r_n\mbox{ for some }r_i\in R\}.
\end{equation}

It is easy to see that reflection length in an affine Coxeter group
is a bounded function. In fact, McCammond and Petersen \cite{McCamm}
proved that for an affine group $W$ acting on the Euclidean space ${\mathbb E}^n$
the maximal value of reflection length is equal to $2n$. They also observed
that reflection length in a free Coxeter group is unbounded and asked
if this is the case for all non-affine Coxeter groups.
The purpose of the present paper is to answer this question
in the positive:

\begin{theorem}
\label{maintheorem}
For any non-affine Coxeter group $(W,S)$ the reflection length
is an unbounded function on $W$.
\end{theorem}

The proof of Theorem \ref{maintheorem} is divided into two parts.
In the first part we exhibit hyperbolic behavior
of \emph{minimal non-affine} Coxeter groups.
Specifically, we prove the following result:

\begin{theorem}
\label{hyperbolicquotient}
Every minimal non-affine Coxeter group admits a surjection onto
a non-elementary word-hyperbolic group.
\end{theorem}

Theorem \ref{hyperbolicquotient}, besides being an important ingredient
of our main result, is interesting in its own right since it gives
a positive answer to a special case of a conjecture of Januszkiewicz
\cite{Gri}, \cite{Jan} that every non-affine Coxeter group
has a non-elementary word-hyperbolic quotient.

The second part of the proof of Theorem \ref{maintheorem} amounts
to the statement that every group with a non-elementary word-hyperbolic
quotient is unbounded in the bi-invariant word metric defined by
any finite set of generators.

The paper is organized as follows. In Section \ref{sectionquotients}
we recall that minimal non-affine Coxeter groups act on hyperbolic
spaces with (possibly unbounded) simplicial fundamental domains,
so that some subgroups of finite index are fundamental groups
of hyperbolic manifolds of finite volume. Then we perform
negatively curved fillings of such hyperbolic manifolds,
generalized to an equivariant setting. This allows to construct
a simply-connected negatively curved space with a geometric action
of a quotient of the Coxeter group, which yields
Theorem \ref{hyperbolicquotient}. We also discuss
a possible generalization of this construction to
all non-affine Coxeter groups. Section \ref{sectionmetric}
is devoted to the proof of Theorem \ref{maintheorem}.
For this purpose we note that it easily reduces
to the special case when the Coxeter group
is minimal non-affine. Then we use the existence
of unbounded quasimorphisms on word-hyperbolic groups
to deduce that any bi-invariant metric on a minimal non-affine
Coxeter group, in particular reflection length, is unbounded.

We assume basic familiarity with Coxeter groups and negative
curvature. For both of these we refer the reader to \cite{Dav}
and \cite{BriHaf}.

\section{Minimal non-affine Coxeter groups}
\label{sectionquotients}

An affine Coxeter group is a direct product of irreducible
spherical and Euclidean reflection groups. A \emph{minimal
non-affine} Coxeter group is a Coxeter group which is itself
non-affine, but all of its proper special subgroups are affine.
Every non-affine Coxeter group has a minimal non-affine
special subgroup; it is just any special subgroup
minimal with respect to inclusion among the non-affine ones.

These groups have the following geometric nature:

\begin{proposition}[\cite{Bou}, V.4, Exercises 12b, 12c, and 13]
\label{minimalnonaffinehn}
Every minimal non-affine Coxeter group $(W,S)$
can be faithfully represented as a discrete reflection group
acting on the hyperbolic space ${\mathbb H}^n$, where $n=|S|-1$
and the elements of $S$ act as reflections with respect to
codimension $1$ faces of a simplex $\sigma$ (some of the vertices
of $\sigma$ might be ideal).
\end{proposition}

From now on we identify $W$ with its image
in $\hbox{Isom}({\mathbb H}^n)$ under this
representation.

As a linear group $W$ has a torsion-free normal subgroup $W'$
of finite index. Any such $W'$ acts freely on ${\mathbb H}^n$
and the quotient ${\mathbb H}^n/W'$ is a complete $n$-dimensional
hyperbolic manifold. If some vertices of $\sigma$
are ideal, then this manifold is not compact and has
finitely many cusps, each having a horoball neighborhood
homeomorphic to a product of an $(n-1)$-dimensional
compact flat manifold and a half-line.

In order to prove Theorem \ref{hyperbolicquotient},
we are going to repeat the procedure described in \cite{MosSag},
which produces word-hyperbolic quotients of fundamental
groups of hyperbolic manifolds. However,
we actually need a quotient of the group $W$,
not just of its subgroup of finite index. For this reason
we will have to review the methods of \cite{MosSag} and
verify at each step that they can be applied
in an equivariant setting.

\begin{proposition}
\label{equivariantmanifold}
There exists a torsion-free normal subgroup
$\overline{W}\triangleleft W$ of finite index,
with the following property:
One can remove open horoball neighborhoods of all cusps
of the hyperbolic manifold ${\mathbb H^n}/\overline{W}$
to obtain a manifold $M$ with boundary composed of
compact flat manifolds, such that:
\begin{enumerate}
\item Any closed geodesic in a component of $\partial M$
has length $>2\pi$.
\item The isometric action of the finite group $F=W/\overline{W}$
on the manifold ${\mathbb H^n}/\overline{W}$ restricts
to an isometric action of $F$ on $M$ preserving $\partial M$.
\end{enumerate}
\end{proposition}

\begin{proof}
The collection of simplices $\{w(\sigma):w\in W\}$
forms a tessellation of ${\mathbb H}^n$ with some ideal vertices,
namely the images of ideal vertices of $\sigma$ under
the action of $W$.

For each ideal vertex of the simplex $\sigma$ take a small open
horoball centered at that vertex, thus obtaining
a collection ${\mathcal B}_0$ of pairwise disjoint horoballs.
Then ${\mathcal B}=\{wB:w\in W,B\in\mathcal B_0\}$
is a $W$-invariant collection of pairwise disjoint
horoballs in ${\mathbb H}^n$ centered at the ideal vertices
of the tessellation. Also, $W$ acts by isometries
on the complement ${\mathbb H}^n\setminus\bigcup{\mathcal B}$.

For every $s\in S$ such that $S\setminus\{s\}$
generates a Euclidean special subgroup $V_s$ of $W$
acting on the horosphere ${\mathbb E}_s=\partial B_s$,
where $B_s\in{\mathcal B}_0$, we will define
a finite set $A_s\subset V_s$.
Fix a simplex $\tau_s\subset{\mathbb E}_s$ fundamental
for the action of $V_s$ on ${\mathbb E}_s$ and let
$$A_s=\{v\in V_s\setminus\{1\}:\hbox{dist}(\tau_s,v(\tau_s))\le 2\pi\}.$$
Since $W$ is virtually torsion-free and residually finite,
we can find a torsion-free normal subgroup
$\overline{W}\triangleleft W$ disjoint with
each of the sets $A_s$.

We claim that $\overline{W}$ is our desired group.

The union $\bigcup{\mathcal B}$ is $W$-invariant,
and so is its boundary composed of horospheres.
Hence we can define the required manifold $M$ as
the quotient of ${\mathbb H}^n\setminus\bigcup{\mathcal B}$
under the action of $\overline{W}$. The isometric
action of $W$ on ${\mathbb H}^n\setminus\bigcup{\mathcal B}$
descends to an isometric action of the finite quotient group
$F=W/\overline{W}$ on the compact quotient manifold $M$,
and clearly $\partial M$ is preserved under that action.
This proves (2).

It remains to show (1). A closed geodesic in a component
of $\partial M$ can be lifted to a segment in the boundary
of a horoball from ${\mathcal B}$ whose endpoints
belong to the same $\overline{W}$-orbit. Therefore it suffices
to verify that for each point $p$ lying on the boundary
$\partial B$ of some horoball $B\in{\mathcal B}$
and for any nontrivial element $w\in\overline{W}$
such that $w(p)\in\partial B$, the distance
between $p$ and $w(p)$ in $\partial B$ satisfies
$d(p,w(p))>2\pi$. But there exists an element $u\in W$
such that $u(B)=B_s$ for some $s\in S$ and the horoball
$B_s\in{\mathcal B}_0$ centered at the ideal vertex of $\sigma$
corresponding to $s$. Moreover, multiplying $u$ by an
element of the special subgroup $V_s$
of $W$ we may assume that $u(p)=q\in\tau_s$. Then we have
$$d(p,w(p))=d(u^{-1}(q),w(u^{-1}(q)))=d(q,u(w(u^{-1}(q))))=d(q,w'(q)),$$
where $w'=uwu^{-1}\in\overline{W}$ as $\overline{W}$ is normal
in $W$. Now $w'$ is nontrivial and $w'\not\in A_s$ for all $s$.
Finally, the definition of the sets $A_s$ implies that
$d(q,w'(q))>2\pi$.
\end{proof}

From now on fix a subgroup $\overline{W}$ as in Proposition
\ref{equivariantmanifold}. The symbols $M$ and $F$ used there
will also retain their meaning.

We are going to recall some portions of the proof of
Negative Curvature Theorem in \cite{MosSag} and
investigate or slightly modify them with the aim
of ensuring $F$-equivariance.

Let ${\mathcal T}$ be the collection of components of $\partial M$.
For each $T\in{\mathcal T}$ choose a number
$r_T\in(-L/2\pi,-1)$, where $L$ is the length of the
shortest closed geodesic on $T$. We can choose the numbers $r_T$
so that whenever $T_1,T_2\in{\mathcal T}$ are isometric,
we have $r_{T_1}=r_{T_2}$. Then for each $T\in{\mathcal T}$
take $T\times[r_T,0]$, collapse $T\times\{r_T\}$ to a point $p_T$
and glue the resulting cone $C(T)$ to $M$ by identifying
$(x,0)\in T\times\{0\}\subset C(T)$ with $x\in T\subset M$.
Let $N$ be the space obtained by gluing all the cones $C(T)$
to $M$. Note that $N$ is a $n$-dimensional pseudomanifold,
and in fact a manifold except at the points $p_T$.

We will put a metric on $N$ so that the isometric action of $F$ on $M$
extends to an isometric action on $N$. On the manifold
$N\setminus\{p_T:T\in{\mathcal T}\}$
this will actually be a Riemannian metric.

Let $ds$ be the Riemannian hyperbolic metric on $M$.
Pick a boundary component $T\in{\mathcal T}$ and let $ds_T$ be
the restriction of $ds$ to $T$; this is a Euclidean
metric on $T$. Moreover, for some small $\epsilon>0$
the metric on $T\times[0,\epsilon)\subset M$ is given by
$$ds^2=e^{2r}ds_T^2+dr^2.$$
Extend this formula for a Riemannian metric to
$T\times(-\epsilon,\epsilon)\subset N$.
Now define a metric on $T\times(r_T,r_T+\epsilon)$ by
$$ds^2=\left(\frac{2\pi}{L}\sinh(r-r_T)\right)^2ds_T^2+dr^2.$$
The metric is now defined on $T\times[(r_T,r_T+\epsilon)\cup(-\epsilon,\epsilon)]$
and on the intersection of this set with $M$ it coincides
with the original metric on $M$. Finally, let the metric
on $T\times(r_T,\epsilon)$ have the form
\begin{equation}
\label{metriconn}
ds^2=f(r)^2 ds_T^2+dr^2,
\end{equation}
where $f$ is a smooth positive increasing convex function such that
(\ref{metriconn}) agrees with the two previous formulas on their domains
(see \cite{MosSag}, Figure 1). Moreover, one can clearly ensure that
whenever $T_1,T_2\in{\mathcal T}$ are isometric, the corresponding
two functions $f$ coincide.

To extend the above metric to the points $p_T$ it suffices to note
that the diameter of $T\times\{r\}$ tends to zero as $r\to r_T$.
Thus we may simply take the completion of the already defined metric
on $N\setminus\{p_T:T\in{\mathcal T}\}$. In this completed metric $N$
is a compact geodesic space.

\begin{proposition}
\label{nisnegativelycurved}
The above construction has the following properties:
\begin{enumerate}
\item The compact space $N$ is negatively curved.
\item Any isometry $\rho:T_1\to T_2$ between (not necessarily distinct)
components $T_1,T_2\in{\mathcal T}$ extends to an isometry $\widehat{\rho}
:C(T_1)\to C(T_2)$. Consequently, the isometric action of $F$ on $M$
extends to an isometric action on $N$.
\end{enumerate}
\end{proposition}

\begin{proof}
(1) is proved in \cite{MosSag}, where it is written
in the case when all manifolds $T\in{\mathcal T}$ are tori,
but it works without any change provided that every $T\in{\mathcal T}$
is a compact flat manifold with no closed geodesics
of length $\le 2\pi$.

To prove (2), we define $\widehat{\rho}(t,r)=(\rho(t),r)$.
Then $\widehat{\rho}$ is constant in the $r$-direction and for every
$r\in[r_{T_1},0]=[r_{T_2},0]$ it induces an isometry between
the sections $T_1\times\{r\}\subset C(T_1)$ and
$T_2\times\{r\}\subset C(T_2)$. The radial nature
of the Riemannian metric (\ref{metriconn}) shows
that $\widehat{\rho}$ is an isometry.
\end{proof}

The cones $C(T)$ are contractible, hence the inclusion
$M\hookrightarrow N$ induces a surjection $\overline{W}=\pi_1(M)\to\pi_1(N)$.
The group $\overline{H}=\pi_1(N)$ can be obtained from $\overline{W}$
by killing all elements of $\overline{W}$ representing a closed geodesic
in some boundary component $T\subset\partial M$.

By Proposition \ref{nisnegativelycurved}(1) $\overline{H}$ is word-hyperbolic.
It is obviously non-elementary; as the fundamental group of a compact
aspherical $n$-dimensional pseudomanifold it has cohomological dimension $n$.
However, in order to prove Theorem \ref{hyperbolicquotient} we have to
lift the virtual surjection $\overline{W}\to\overline{H}$ to an actual
surjection $W\to H$.

\begin{proof}[Proof of Theorem \ref{hyperbolicquotient}]
The inclusion $M\hookrightarrow N$ can be lifted to a map
of the universal covers $\Psi:\widetilde{M}\to\widetilde{N}$.
Note that $\widetilde{M}$ is just ${\mathbb H}^n$ with
a collection of disjoint open horoballs removed.
The symbols $\pi_M$ and $\pi_N$ will denote
the projections $\widetilde{M}\to M$ and $\widetilde{N}\to N$,
respectively. The projection $\pi_N$ restricted to
$\pi_N^{-1}(M)$ is a covering map over $M$.
Similarly for $\pi_N^{-1}(N\setminus M)$;
however, since $N\setminus M$ is a union of the interiors
of the cones $C(T)$ which are simply-connected,
the preimage $\pi^{-1}(N\setminus M)$ is a union
of copies of the interiors of the cones $C(T)$.
It follows that every path joining two points of $\pi_N^{-1}(M)$
is homotopic relative to endpoints to a path contained
in $\pi_N^{-1}(M)$ --- any segment of such a path contained
in a copy of the interior of some cone $C(T)$
can be homotoped to its base.
Consequently, $\pi_N^{-1}(M)$ is connected.
In general $\Psi$ is always a covering map over
a connected component of $\pi_N^{-1}(M)$, so
in our case the image of $\Psi$ is precisely $\pi_N^{-1}(M)$.

The finite group $F=W/\overline{W}$ acts by isometries
on $M$. It follows that $\widetilde{M}$ is invariant
under the action of $W$ on ${\mathbb H}^n$. In other words,
$W$ might be thought of as the group of those
isometries of $\widetilde{M}$ that are lifts of
isometries of $M$ belonging to $F$.

By Proposition \ref{nisnegativelycurved}(2) every element $f\in F$
also acts as an isometry of $N$, so it can be lifted to
an isometry $\widetilde{f}$ of $\widetilde{N}$. Of course,
such a lift is not unique and all possible lifts $\widetilde{f}$
form a $\overline{H}$-coset in $\hbox{Isom}(\widetilde{N})$. Let $H$
be the group of all isometries of $\widetilde{N}$ that are lifts
of some isometry $f\in F$ of $N$. Since the action of $F$
on $N$ is effective (it was already effective on the subspace
$M\subset N$), we see that $H/\overline{H}\cong F$.
Hence $H$ contains a word-hyperbolic subgroup $\overline{H}$
of finite index and thus $H$ is itself word-hyperbolic.
All that is left to show is that there exists a surjection
$\eta:W\to H$.

To this end, consider $w\in W$ as an isometry of $\widetilde{M}$.
We wish to define $\eta(w)\in H$. For this purpose note that
$f=w\overline{W}\in F$ is an isometry of $M$, and therefore
also an isometry of $N$, sending $\pi_M(p)$ to $\pi_M(w(p))$
for any point $p\in\widetilde{M}$. Moreover, since the map
$\Psi:\widetilde{M}\to\widetilde{N}$ covers the inclusion
$M\hookrightarrow N$, we have $\pi_N(\Psi(p))=\pi_M(p)\in N$
and $\pi_N(\Psi(w(p)))=\pi_M(w(p))\in N$. Hence, for a fixed
point $p\in\widetilde{M}$, the isometry $f$ of $N$ can be lifted
to the unique isometry $\widetilde{f}$ of $\widetilde{N}$ satisfying
$$\widetilde{f}(\Psi(p))=\Psi(w(p)).$$
Note that the lift $\widetilde{f}$ in fact does not depend on $p$.
This is because $\Psi(p)$ and $\Psi(w(p))$ in the above equality
vary continuously with $p$ and for any $x\in\widetilde{N}$ the set of
all $\widetilde{f}(x)$ for all lifts $\widetilde{f}$ of $f$
is discrete in $\widetilde{N}$. Hence if we let $\eta(w)=\widetilde{f}$,
then we obtain a function $\eta:W\to H$ such that
$$\eta(w)(\Psi(m))=\Psi(w(m))$$
for all $w\in W$ and $m\in\widetilde{M}$.
Thus for any $w_1,w_2\in W$ we have
$$\eta(w_1w_2)(\Psi(m))=\Psi(w_1(w_2(m)))=
\eta(w_1)(\Psi(w_2(m)))=\eta(w_1)\eta(w_2)(\Psi(m)),$$
and since $\eta(w_1w_2)$ and $\eta(w_1)\eta(w_2)$ are lifts
of the same isometry of $M$, namely $w_1w_2\overline{W}=
w_1\overline{W}\cdot w_2\overline{W}\in F$,
we conclude that $\eta$ is a homomorphism.

It remains to verify that $\eta$ is surjective. Take $\xi\in H$
and a point $q\in\pi_N^{-1}(M)$. Then $\xi$ is a lift
of an isometry $f\in F$ of $N$ (hence also of $M$)
sending $\pi_N(q)\in M$ to $\pi_N(\xi(q))\in M$.
Note that $\xi(q)\in\pi_N^{-1}(M)$, hence
there exist points $p,p'\in\widetilde{M}$
such that $\Psi(p)=q$ and $\Psi(p')=\xi(q)$.
In this situation an element $w\in W$ sent to $\xi\in H$
by $\eta$ can be defined as the lift of $f$ 
satisfying $w(p)=p'$. Such a lift exists
because $\pi_M(p)=\pi_N(\Psi(p))=\pi_N(q)$
and $\pi_M(p')=\pi_N(\Psi(p'))=\pi_N(\xi(q))
=f(\pi_N(q))$.

The proof is now finished.
\end{proof}

\begin{remark}
{\rm
The above construction can be described in the language
of simplices of groups (see \cite{BriHaf}, II.12
for background on complexes of groups). Note that
$W$ is the fundamental group of the $n$-simplex of groups
with the following data: the vertices are indexed by $S$
and the local group assigned to the simplex spanned by $T\subset S$
is the special subgroup of $W$ generated by $S\setminus T$.
In our case all local groups are finite except possibly
at the vertices, where some local groups might be
irreducible affine. The group $H$ is constructed
by taking a sufficiently large finite quotient of $W$
and replacing each local group with its image
in this quotient. Note that those local groups
that were already finite are not affected.

One is tempted to perform this construction for an arbitrary
non-affine Coxeter group $(W,S)$. Indeed, let
$S'\subset S$ be a subset generating a minimal non-affine
special subgroup of $W$. Let $\pi:W\to D$ be
a surjection onto a sufficiently large finite group
and define a simplex of groups: its vertices are indexed
by $S'$ and the local group associated to $T\subset S'$
is $\pi(\langle T\cup S\setminus S'\rangle)$.
This simplex is developable, since it admits a homomorphism
to $D$ which is injective on the local groups.
The natural question is whether the fundamental group
of this simplex is word-hyperbolic. A positive
answer would resolve the conjecture of Januszkiewicz.
However, this is difficult even in the case when $S'$ generates
a hyperbolic triangular group, since one needs
an understanding of finite quotients and those supplied
by the residual finiteness of linear groups do not seem
to serve this purpose well enough.
}
\end{remark}

\section{Hyperbolic quotients and bi-invariant word metrics}
\label{sectionmetric}

\begin{remark}
\label{reductiontominimal}
{\rm
In order to prove that reflection length is unbounded in all
non-affine Coxeter groups, we may specialize to minimal
non-affine groups. Indeed, let $(W,S)$ be a non-affine
Coxeter group and let $(W',S')$ be any of its minimal non-affine
special subgroups. Then by Corollary 1.4 in \cite{Dye}
the restriction of $||\cdot||_R$ defined by (\ref{refllengthdef})
to the subgroup $W'$ coincides with the reflection length
$||\cdot||_{R'}$ defined analogously to (\ref{refllengthdef})
for $(W',S')$ considered as a Coxeter group in its own right.
Hence if $||\cdot||_{R'}$ is unbounded, then $||\cdot||_R$
is unbounded, too.
}
\end{remark}

A quasimorphism of a group $G$ is a function $\varphi:G\to{\mathbb R}$
such that
$$\sup_{g,h\in G}|\varphi(gh)-\varphi(g)-\varphi(h)|<\infty.$$
A quasimorphism is homogeneous if $\varphi(g^n)=n\varphi(g)$
for all $g\in G$ and all integers $n$. Note that a non-zero
homogeneous quasimorphism is unbounded.

The following fact is an elementary application of
cohomology and bounded cohomology (see \cite{Cal},
\cite{EpsFuj}) to word-hyperbolic groups.

\begin{proposition}
\label{hyperbolichavequasihom}
For any non-elementary word-hyperbolic group $G$ there exists
an unbounded homogeneous quasimorphism $\varphi:G\to{\mathbb R}$.
\end{proposition}

\begin{proof}
The bounded cohomology $H^2_b(G;{\mathbb R})$ is infinite-dimensional
(\cite{EpsFuj}, Theorem 1.1). On the other hand, $H^2(G;{\mathbb R})$
is finite-dimensional since $G$ is finitely presented. Moreover,
we have an exact sequence
$$0\to H^1(G;{\mathbb R})\to Q(G)\to H^2_b(G;{\mathbb R})\to H^2(G;{\mathbb R})$$
(\cite{Cal}, Theorem 2.50), where $Q(G)$ is the vector space
of homogeneous quasimorphisms of $G$. It follows that $Q(G)$
is not only non-trivial, but actually infinite-dimensional.
\end{proof}

For a group $G$ generated by a finite set $X$ we can define
a bi-invariant word metric on $G$ in the following way.
Let $X^*=\{gxg^{-1}:x\in X,g\in G\}$ and let
$$||g||_{X^*}=\min\{n:g=g_1g_2\ldots g_n
\hbox{ for some }g_i\in X^*\}.$$
Note that (\ref{refllengthdef}) is an example of a
bi-invariant metric. The bi-invariance means here
that $||\cdot||_{X^*}$ is invariant
under conjugation: $||kgk^{-1}||_{X^*}
=||g||_{X^*}$ for all $k,g\in G$.

If $f:G_1\to G_2$ is a surjective homomorphism
and $G_1$ is generated by a finite set $X_1$,
then $G_2$ is generated by $X_2=f(X_1)$ and
the associated bi-invariant word metrics satisfy
$$||f(g)||_{X_2^*}\le||g||_{X_1^*}$$
for every $g\in G_1$. In particular, if
$||\cdot||_{X_2^*}$ is unbounded,
then so is $||\cdot||_{X_1^*}$.

\begin{proposition}[\cite{GalKed}, Lemma 3.7]
\label{quasihominequality}
Let $G$ be a group generated by a finite set $X$
with the associated bi-invariant word metric
$||\cdot||_{X^*}$.
Then for any homogeneous quasimorphism $\varphi:G\to{\mathbb R}$
there exists a constant $C>0$ such that
$$||g||_{X^*}\ge C\cdot\varphi(g)$$
for every $g\in G$.
\end{proposition}

\begin{corollary}
\label{quasihomtobiinvariant}
Suppose that a finitely generated group admits an unbounded
homogeneous quasimorphism. Then any bi-invariant word metric
on this group is unbounded.
\end{corollary}

We are now ready to deduce the main result of the paper.

\begin{proof}[Proof of Theorem \ref{maintheorem}]
Let $(W,S)$ be a non-affine Coxeter group. By Remark \ref{reductiontominimal}
we may assume that $W$ is minimal non-affine. Theorem \ref{hyperbolicquotient}
supplies a surjection of $W$ onto a non-elementary word-hyperbolic group $H$.
Proposition \ref{hyperbolichavequasihom} and Corollary \ref{quasihomtobiinvariant}
imply that any bi-invariant word metric on $H$ is unbounded. As explained
before Proposition \ref{quasihominequality}, it follows that any bi-invariant
word metric on $W$ is unbounded. Since reflection length is one of these metrics,
the proof is complete.
\end{proof}

{\bf Acknowledgments.}
I would like to thank Jan Dymara, {\'S}wiatos\l aw R. Gal,
and Tadeusz Januszkiewicz for patient explanations and
helpful conversations.

\end{document}